\documentclass{article}
\usepackage{amsmath}
\usepackage{amsfonts}
\usepackage{mathrsfs}
\usepackage{amsthm, amssymb}
 \newtheorem{thm}{Theorem}[section]
\newtheorem{rem}[thm]{Remark}
\newtheorem{defi}[thm]{Definition}
\newtheorem{nota}[thm]{Notation}
\newtheorem{lem}[thm]{Lemma}
 \newtheorem{prop}[thm]{Proposition}
\newtheorem{cor}[thm]{Corollary}

 \newtheorem{conj}[thm]{Conjecture}
 \newtheorem{prob}[thm]{Problem}
\def\f#1{{\mathbb{F}}_{#1}}

\begin{document}
\title{Rational points on complete symmetric hypersurfaces over finite fields}
%\title{A Computational Problem on a Generalized Vandermonde Determinant}
\author{Jun Zhang\footnote{Jun Zhang, School of Mathematical Sciences, Capital Normal University, Beijing 100048, P.R. China. E-mail: junz@cnu.edu.cn} \and Daqing Wan
	\thanks{Daqing Wan, Department of Mathematics, University of California, Irvine, CA 92697, USA. Email: dwan@math.uci.edu}}

\date{}
\maketitle
\begin{abstract}
For any affine hypersurface defined by a complete symmetric polynomial in $k\geq 3$ variables of degree $m$ over the finite field $\f{q}$ 
of $q$ elements, a special case of our theorem says that this hypersurface  has at least $6q^{k-3}$ rational points over $\f{q}$ if $1\leq m \leq q-3$ and $q$ is odd. A key ingredient in our proof is Segre's classical theorem on ovals in finite projective planes. 

{\bf Keywords: }Rational point, complete symmetric polynomial, generalized Vandermonde determinant, Reed-Solomon code.
\end{abstract}

\section{Introduction}

Let $\f{q}$ denote the finite field of $q$ elements with characteristic $p$. 
The study of $\f{q}$-rational points on a hypersurface defined by a symmetric 
polynomial over $\f{q}$ has many important applications. The connection with coding theory and finite geometry will become clear 
later in this paper. Another recent example is given in ~\cite{BB20}, where a 
family of symmetric hypersurfaces with many points is used to construct minimal codes from
cutting blocking sets. 

There are three classes of symmetric polynomials introduced by Newton. These are 
power sum symmetric polynomials (Fermat hypersurfaces), elementary symmetric polynomials 
and complete symmetric polynomials.  The first two classes have been studied extensively in number theory. 
In this paper, we apply tools from coding theory and finite geometry to 
investigate the third class, namely,  the complete symmetric 
polynomials as defined below. 

\begin{defi} 
	The homogeneous complete symmetric polynomial of degree $m$ in the $k$-variables $\{x_1,x_2,\cdots,x_k\}$ 
	is defined by 
	$$h_m(x_1,x_2,\cdots,x_k)=\sum_{1\leq i_1\leq i_2\leq \cdots \leq i_m\leq k}x_{i_1}x_{i_2}\cdots x_{i_m}.$$
	%=\sum_{j_1+\cdots+j_k=m,\,j_i\geq 0}x_1^{j_1}\cdots x_k^{j_k}.$$
	%Namely, $h_m(x_1,..., x_k)$ is the coefficient of $t^m$ in the power series 
	%$$\frac{1}{(1-x_1t)\cdots (1-x_kt)} =\sum_{m=0}^{\infty} h_m(x_1,x_2,\cdots,x_k)t^m.$$ 
\end{defi}
By definition, we have $h_0(x_1,x_2,\cdots,x_k)=1$,   
$$h_1(x_1,x_2,\cdots,x_k)=x_1+x_2+\cdots+x_k,$$
$$h_2(x_1,x_2,\cdots,x_k)=\sum_{i=1}^kx_i^2+\sum_{1\leq i< j\leq k}x_ix_j,$$
etc.  Just like the elementary symmetric polynomials, the complete symmetric polynomials 
$h_m(x_1,..., x_k)$ ($0\leq m\leq k)$ generate the algebra of all symmetric polynomials in $k$-variables 
over $\mathbb{Z}$. In characteristic zero, the projective hypersurface defined by $h_m(x_1,..., x_k)=0$ is smooth for all $k\geq 2$. 
In characteristic $p>0$, the singular locus (even its size) of  the projective hypersurface defined by $h_m(x_1,..., x_k)=0$ is unknown. 

\begin{defi} 
	A complete symmetric polynomial of degree $m$ over $\f{q}$ in the $k$-variables $\{x_1,x_2,\cdots,x_k\}$ is defined as
	$$h(x_1,..., x_k):=\sum_{e=0}^m a_eh_e(x_1,x_2,\cdots,x_k),$$
	%=\sum_{e=0}^m a_e\sum_{j_1+j_2+\cdots+j_k=e,\ j_i\geq 0}x_1^{j_1}x_2^{j_2}\cdots x_k^{j_k},$$
	where $a_e\in \f{q}$ and $a_m\not=0$. 
\end{defi}
Thus, a complete symmetric polynomial 
in $k$-variables is simply a linear combination of the homogeneous complete symmetric polynomials in $k$-variables.  
Equivalently, a complete symmetric polynomial $h(x_1,..., x_k)$ is simply a polynomial in $k$-variables where all terms of the 
same total degree have the same coefficients. We stress that such polynomials are not homogeneous in general. 
We are interested in the number of $\f{q}$-rational points on the 
affine hypersurface defined by a complete symmetric polynomial $h(x_1,..., x_k)$ over $\f{q}$.  
As noted above, the singular locus (even its size) of the affine hypersurface defined by 
$h(x_1,..., x_n) =0$ can be quite complicated, especially in characteristic $p$.  

\begin{nota} 
	Let $h(x_1,..., x_k)$ be a complete symmetric polynomial of degree $m$ in $k$-variables over $\f{q}$. Let
	$$N_q(h):= \#\{(x_1,..., x_k)\in \f{q}^k\, |\, h(x_1,..., x_k)=0\},$$
	denote the number of $\f{q}$-rational points on the 
	affine hypersurface defined by $h(x_1,..., x_k)=0$. 
\end{nota}

Our basic problem is to study when $N_q(h)>0$ and to give a good lower bound when it is positive. 
The problem is trivial if $m=0$ and thus $h$ is a constant. 
We shall assume that $m>0$ and so $h$ is not a constant. 
A consequence of our main theorem is the following result.

\begin{thm}\label{thm1} Let $h(x_1,..., x_k)$
	be a complete symmetric polynomial in $k\geq 3$ variables over $\f{q}$ of 
	degree $m$ with $1\leq m\leq q-3$.  If $q$ is odd, then 
	$$N_q(h)\geq 6q^{k-3}.$$ 
	\end{thm}
\begin{rem} Let us compare this theorem with the relevant results in the literature.  
 A classical result of Waring \cite{War35} implies that if $N_q(h)>0$, then 
$N_q(h)\geq q^{k-m}$. This is apparently weaker than $6q^{k-3}$ if $m\geq 3$,  and trivial 
if $k\leq m$. 
The condition  $N_q(h)>0$ itself is highly non-trivial to check unless $h$ has no constant term. 
If one applies Deligne's theorem 
on the Weil conjecture, even in the sufficiently smooth case (the size of singular locus 
is already unknown), one would need 
to assume that the degree $m$ is small compared to $q$ in order to prove a 
non-trivial lower bound for $N_q(h)$. One would at least need 
something like $m =O(q^{\frac{1}{2}-\epsilon_k})$,  
where $\epsilon_k$ is a positive constant depending on $k$. 
If $k>m$, the classical Chevally-Waring-Ax-Katz type  
theorem implies $N_q(h)$ is divisible by $q^{\lceil \frac{k-m}{m}\rceil}$, see \cite{W95} for simple proofs of various such divisibility results.  Again, one needs to assume both $N_q(h)>0$ and $k>m$ in order to 
derive a non-trivial lower bound for $N_q(h)$. Our theorem above has several new features. 
It does not assume that the degree $m$ is small compared to $q$. It does not assume that  
$N_q(h)>0$ either. The lower bound $6q^{k-3}$ works for all degree $1\leq m \leq q-3$. 
When $m\geq q-2$, the problem becomes more complicated as $m$ grows. 
But as we shall see, a stronger version of the problem (with distinct coordinate rational points) 
in the large degree $m$ case can be reduced to the smaller degree $m<q$ case. 
\end{rem}
\begin{rem}
We note that the condition $k\geq 3$ in the theorem cannot be dropped. For instance, if $k=2$, 
one checks that 
$$h_m(x_1, x_2) = \frac{x_1^{m+1} - x_2^{m+1}}{x_1-x_2}.$$
If $(m+1, p(q-1))=1$, then the only $\f{q}$-rational 
point of $h_m(x_1, x_2)=0$ is the origin and so $N_q(h_m)=1$. 
Taking $k=2, q=5, m=2$, one finds that $N_q(h_2(x_1, x_2)) =1 < 6/5 = 6 q^{2-3}$. 
%If one further takes $m$ such that 
%$(m, q-1)>1$, then $N_q(T_m-\alpha) =0$ for any $\alpha\in \f{q}^*$ such that $\alpha/(m+1)$ is not an $(m, q-1)$-th  
%power in $\f{q}$. 
	
\end{rem}

For even $q$, the problem is more subtle. We have the following conjecture giving a slightly weaker bound. 

\begin{conj}\label{thm2} Let $h(x_1,..., x_k)$
	be a complete symmetric polynomial in $k\geq 4$ variables over $\f{q}$ of 
	degree $m$ with $1\leq m\leq q-4$.  If $q$ is even, then 
	$$N_q(h)\geq 24q^{k-4}.$$ 
	\end{conj}

For even $q$, unconditionally, we have the following significantly weaker result. 

\begin{thm}\label{thm3}Let $h(x_1,..., x_k)$
	be a complete symmetric polynomial in $k$ variables over $\f{q}$ of 
	degree $m$ with $1\leq m\leq q/2$.  If $q\geq 8$ is even and $k\geq q/2$, then 
	$$N_q(h)\geq \left({{q}\over {2}}\right)! \cdot q^{k-{{q}\over 2}}.$$ 
	\end{thm}

%It would also be interesting to ask if a similar theorem  holds for 
%some other family of symmetric polynomials in $k$ variables of degree $m$, such as 
%the elementary symmetric polynomials.  

The paper is organized as follows. In Section~\ref{distinctcoordinates}, 
we will consider the stronger question on the number of $\f{q}$-rational 
points with distinct coordinates.  In Section~\ref{proofoftheorem2.5}, we relate the complete symmetric 
polynomial $h(x_1,..., x_k)$ to the determinant of certain generalized 
Vandermonde determinant. It is shown that the existence of $\f{q}$-rational points with distinct coordinates is equivalent to 
the vanishing of certain generalized Vandermonde determinant.  This latter problem is further reduced to 
the classification of deep holes for Reed-Solomon codes, equivalently possible MDS extension of Reed-Solomon codes.  
Theorem \ref{thm1} then follows by applying the classical result $(k=3$, $p$ odd) of Segre \cite{Seg55}  on ovals in finite projective planes. 
In section $4$, we consider the harder problem of solutions of $h(x_1,..., x_k)$ 
where the variables $x_i$'s only varies in a subset $S$ of $\f{q}$.

Segre's old result  is now a special case of the Cheng-Murray conjecture \cite{CM07} which classifies deep holes for Reed-Solomon codes. The latter is in turn a consequence of the normal rational curve conjecture in finite geometry. 
The Cheng-Murray conjecture remains open in general. But it has been proved by Zhuang-Cheng-Li \cite{ZCL16} 
in the case $k\leq p$  and later by Kaipa \cite{Kaipa17} in the case $k\geq [(q+1)/2]$.  
These recent works will give us additional results on the number of rational points with distinct coordinates, see the second part of section $3$. 
To be self-contained, we also include a simpler and more direct 
proof of these more general results.

\section{Rational points with distinct coordinates}\label{distinctcoordinates}
In coding theory, one often requires 
the additional condition that the coordinates of the rational point are distinct.

\begin{nota} 
	Let $h(x_1,..., x_k)$ be a complete symmetric polynomial of degree $m$ in $k$-variables over $\f{q}$. Let
	$$N_q^*(h):=\# \{(x_1,..., x_k)\in \f{q}^k\,|\, h(x_1,..., x_k)=0\,\mbox{and}\, ~
	x_i\not= x_j ~
	 \forall\  i\not=j\}$$
	denote the number of $\f{q}$-rational points on the 
	affine hypersurface defined by $h(x_1,..., x_k)=0$  
	with the additional condition that the coordinates are distinct. 
\end{nota}

We are interested in when $N_q^*(h)\geq 1$. Since $T$ is symmetric, 
$N_q^*(h)\geq 1$ is equivalent to $N_q^*(h)\geq k!$ by permutations of the solutions. 
Our main problem is the following

%The degree restriction $m\leq q-3$ in the theorem can be weakened. We have  

\begin{conj}\label{conj:zeroofcompsymmpoly}  
	Let $3\leq k\leq q-2$ ($4\leq k\leq q-3$ if $q$ is even). Let 
	$$h(x_1,..., x_k)=\sum_{e=0}^m a_eh_e(x_1,..., x_n)\in \f{q}[x_1,..., x_k]$$
	be a complete symmetric polynomial of positive 
	degree $m$.  Then $N_q^*(h)\geq 1$ if and only if the reduction $x^{k-1}\left(\sum_{e=0}^m a_ex^{e}\right)$ 
	modulo $(x^q-x)$ is not a polynomial of degree equal to $k-1$. \end{conj}

As a consequence, if the condition of the conjecture holds, then we would have 
$$N_q(h)\geq N_q^*(h)\geq k!.$$ 
 
We shall see that the reduction condition modulo $(x^q-x)$ is necessary in order for  $N_q^*(h)\geq 1$. The difficulty lies 
in the sufficient part of the condition.  The reduction condition is also simple to check. For $0\leq j\leq q-2$, let 
$$b_j = \sum_{e\equiv j \bmod {(q-1)}} a_e.$$
Then, 
$$x^{k-1}\left(\sum_{e=0}^m a_ex^{e}\right)\equiv \sum_{j=0}^{q-k} b_j x^{j+k-1} 
+ \sum_{j=q-k+1}^{q-2}b_j x^{j+k-q}  \mod (x^q-x),$$
where the second sum on the right is a polynomial of degree at most $k-2$.  
Thus, the reduction	is not a polynomial of degree equal to $k-1$ 
	if and only if either $b_{0} =0$ or that $b_j$ is not zero 
	for some $1\leq j\leq q-k$. As an example, if $1\leq m\leq q-k$, then $g(x)=x^{k-1}\left(\sum_{e=0}^m a_ex^{e}\right)$ 
is a polynomial of positive degree $m+k-1\leq q-1$, and so its reduction 
modulo $(x^q-x)$ has the same positive degree $m+k-1$ which is not equal to $k-1$.  
A special case of the above conjecture is then the following 

\begin{conj}\label{cor:zeroofcompsymmpoly}  
	Let $3\leq k\leq q-2$ ($4\leq k\leq q-3$ if $q$ is even). Let 
	$$h(x_1,..., x_k)=\sum_{e=0}^m a_eh_e(x_1,..., x_n)\in \f{q}[x_1,..., x_k]$$
	be a complete symmetric polynomial of positive 
	degree $m$.  If $1\leq m\leq q-k$,	 then 
	$N_q(h)\geq N_q^*(h)\geq k!$. 
\end{conj}
\begin{rem}
The simple condition $3\leq k\leq q-2$ cannot be improved. Since for $k=1$, $h(x)$ 
can be an arbitrary univariate polynomial of degree $m\geq 2$ 
and one can easily find one such $h(x)$ (any irreducible $h(x)$ will do) 
such that $N_q^*(h)=0$. If $k=2$, take $f(x)=\sum_{e=1}^{m+1} a_ex^e$ to be a 
permutation polynomial of degree $m+1$ over $\f{q}$, then $h(x_1, x_2): = (f(x_1)-f(x_2))/(x_1-x_2)$ 
is a complete symmetric polynomial of degree $m$ with no $\f{q}$-rational points off the diagonal $x_1=x_2$. 
The condition $k\leq q-2$ is optimal too. 
For instance, if $k=q$, there is only one possibility (up to permutation) for solutions with distinct coordinates. One can easily modify the 
constant term of $h$ so that $N_q^*(h)=0$. If $k=q-1$ ($q>3$ odd), a solution set 
$\{\alpha_1,\cdots, \alpha_{q-1}\}$ is equal to $\f{q}-\{\alpha\}$ for some $\alpha \in \f{q}$, 
one then checks that  
$$\sum_{i=1}^{q-1} \alpha_i = -\alpha,\ \sum_{i=1}^{q-1} \alpha_i^2 =-\alpha^2.$$
Then,
$$2h_2(\alpha_1,\cdots, \alpha_{q-1}) = (\sum_{i=1}^{q-1}\alpha_i)^2 +\sum_{i=1}^{q-1}\alpha_i^2 = \alpha^2 -\alpha^2 =0$$
for all distinct $\alpha_1,\cdots, \alpha_{q-1}$ in $\f{q}$. As a consequence, 
the complete symmetric polynomial $2h_2(x_1,..., x_{q-1}) +c$ 
($c\in \f{q}^*$) has no $\f{q}$-rational points with distinct coordinates. 
 	
\end{rem} 

The aim of this paper is to prove that the above conjecture is true 
if either $k\leq p$ (this is always satisfied if $q=p$ is a prime) or if $k\geq \lfloor (q+1)/2\rfloor$. 

\begin{thm}\label{thm4}
	 Let $3\leq k\leq q-2$ ($4\leq k\leq q-3$ if $q$ is even). 
	Let $$h(x_1,..., x_k)=\sum_{e=0}^m a_eh_e(x_1,..., x_n)\in \f{q}[x_1,..., x_k]$$
	be a complete symmetric polynomial of positive 
	degree $m$.  Assume either 
	$k\leq p$  or $k\geq  \lfloor (q+1)/2\rfloor$.  Then $N_q^*(h)\geq 1$ if and only if the reduction $x^{k-1}\left(\sum_{e=0}^m a_ex^{e}\right)$ 
	modulo $(x^q-x)$ is not a polynomial of degree equal to $k-1$.\end{thm}

By using Theorem~\ref{thm4}, we now give the proofs of Theorems~\ref{thm1} and~\ref{thm3}.
\begin{flushleft}
	\textbf{Proofs of Theorems~\ref{thm1} and~\ref{thm3}.} 
	By definition, one checks that 
	$$h_m(x_1,..., x_k) = \sum_{e=0}^m h_e(x_1, x_2, x_3) h_{m-e}(x_4, \cdots, x_k).$$
	It follows that if $h(x_1,..., x_k)$ is a complete symmetric polynomial in $k$-variables of degree $m$, then 
	for every choice of $(a_4,\cdots, a_k) \in \f{q}^{k-3}$, the specialization $h(x_1, x_2, x_3, a_4, \cdots, a_k)$  
	is a complete symmetric polynomial of the same degree $m$ in the $3$ variable $\{x_1, x_2, x_3\}$. We apply the case 
	$k=3$ of the above theorem which is true when $q$ is odd and $1\leq m \leq q-3$.  This proves 
	that 
	$$N_q(h) \geq \sum_{a_4,\cdots, a_k \in \f{q}} N_q^*(h(x_1, x_2, x_3, a_4, \cdots, a_k)) \geq 6q^{k-3}.$$
	Note that for $q$ even, the same argument shows that the case $k=4$ of the above conjecture implies 
	$$N_q(h)\geq 24q^{k-4}$$
	if both $k\geq 4$ and $1\leq m \leq q-4$. But the case $k=4$ ($q$ even) of the above conjecture is still open.  For $q$ even, we use the case $k=[(q+1)/2]=q/2$ of the above 
	theorem to deduce the weaker Theorem \ref{thm3}. 
\end{flushleft}

%\section{Related to a finite geometry problem}
\section{The proof of Theorem~\ref{thm4}}\label{proofoftheorem2.5}
In this section, we give a proof of Theorem~\ref{thm4}. We first translate Theorem~\ref{thm4} to a vanishing problem of a generalized Vandermonde determinant. The latter is further reduced to the classification problem of deep holes for Reed-Solomon codes. We can then apply results from coding theory and finite geometry. 

\subsection{Generalized Vandermonde determinant}
Let $k$ be an integer such that $2\leq k\leq q$. For any polynomial $f(x)\in \f{q}[x]$ let $M_f$ denote the $k\times k$ matrix of polynomials
$$	M_f(x_1,\cdots,x_k)=\left(
\begin{array}{cccc}
1 & 1 & \cdots & 1  \\
x_1 & x_2 & \cdots & x_k  \\
\vdots & \vdots & \ddots & \vdots  \\
x_1^{k-2} & x_2^{k-2} & \cdots & x_k^{k-2} \\
f(x_1) & f(x_2) & \cdots & f(x_k )\\
\end{array}
\right). $$
Let 
$$D_f(x_1,\cdots,x_k)=\det 	M_f(x_1,\cdots,x_k)$$
denote its determinant.  An interesting problem is to decide when the determinant $D_f(\alpha_1,\cdots,\alpha_k)$ is non-zero for all pairwise distinct $\alpha_1,\cdots,\alpha_k\in \f{q}$. Since $\alpha^q=\alpha$ for all $\alpha\in \f{q}$, reducing 
$f(x)$ modulo $(x^q-x)$ if necessary, we can assume that $\deg(f) \leq q-1$. 

It is obvious that $D_f(x_1,\cdots,x_k)$ is symmetric with respect to the variables $x_1,\cdots, x_k$. Indeed, there is an explicit formula for the determinant $D_f$ in terms of complete symmetric polynomials.  	

\begin{prop}[\cite{Fink73}]\label{fink}
	For any polynomial $f(x)=\sum_{i=0}^{d}a_ix^i\in\f{q}[x]$ with $a_d\neq 0$, define 
	$$C_f(x_1,\cdots, x_k):=\sum_{i=k-1}^{d}a_ih_{i-(k-1)}(x_1,\cdots, x_k).$$
	This is a complete symmetric polynomial of degree $d-k+1$, which depends only on the degree at least $k-1$ part of $f(x)$. 
	Then, we have
	\[
	D_f(x_1,\cdots, x_k)=C_f(x_1,\cdots, x_k)\prod_{1\leq i<j\leq k}(x_j-x_i). 
	\]
\end{prop}
\begin{cor}
	For any polynomial $f(x)=\sum_{i=0}^{d}a_ix^i\in\f{q}[x]$ with $a_d\neq 0$, and any pairwise distinct $\alpha_1,\cdots,\alpha_k\in \f{q}$, we have
	\[
	D_f(\alpha_1,\cdots, \alpha_k)= 0\,\,\mbox{if and only if}\,\,C_f(\alpha_1,\cdots, \alpha_k)=0.
	\]
\end{cor}

Let $N_q^*(C_f)$ denote the number of  $\f{q}$-rational points of $C_f$ with distinct coordinates, 
and let  $N_q^*(D_f)$ denote  the number of $\f{q}$-rational points of $D_f$ with distinct coordinates. 
The above corollary says that 
$$N_q^*(C_f) =N_q^*(D_f)$$
for all polynomial $f(x)\in \f{q}[x]$.  
Conversely, given a complete symmetric polynomial  
$$h(x_1,..., x_k)=\sum_{e=0}^m a_eh_e(x_1,..., x_n)\in \f{q}[x_1,..., x_k],$$
our construction shows that 
$$h(x_1,...x_k) = C_g(x_1,..., x_k), \  g(x) =x^{k-1}\left(\sum_{e=0}^m a_ex^{e}\right).$$
So in this way, counting the $\f{q}$-rational points of a complete symmetric polynomial with distinct coordinates is equivalent to counting the $\f{q}$-rational points of the corresponding determinant $D_f$ with distinct coordinates.

 By the above discussion, Theorem~\ref{thm4} can be translated to the following theorem.
\begin{thm}\label{thm5}
	Let $3\leq k\leq q-2$ ($4\leq k\leq q-3$ if $q$ is even). For any complete symmetric polynomial of positive 
degree $m$
$$h(x_1,..., x_k)=\sum_{e=0}^m a_eh_e(x_1,..., x_n)\in \f{q}[x_1,..., x_k],$$
let $g(x)$ be the reduction of $x^{k-1}\left(\sum_{e=0}^m a_ex^{e}\right)$ modulo $(x^q-x)$. Assume either 
$k\leq p$  or $k\geq  \lfloor (q+1)/2\rfloor$.  Then $N_q^*(D_g)\geq 1$ if and only if $\deg g(x)\neq k-1.$ 
\end{thm} 

Note that if $\deg g(x)=k-1$, saying $g(x)=ax^{k-1}+g_1(x)$ for some $a\in \f{q}^*$ and $g_1(x)\in \f{q}[x]$ with $\deg g_1(x)\leq k-2$, then it is easy to see that $D_g=a\prod_{1\leq i<j\leq k}(x_j-x_i). $ So $N_q^*(D_g)=0.$ This proves one direction of Theorem~\ref{thm5} (and Conjecture~\ref{conj:zeroofcompsymmpoly}). Next, we focus on the other direction of Theorem~\ref{thm5}.
\subsection{Reed-Solomon codes and their MDS extensions}
In this subsection, we further reduce Theorem~\ref{thm5} in the previous subsection to the classification of deep holes for Reed-Solomon codes, 
equivalently MDS extension of Reed-Solomon codes.  

Suppose the finite field $\f{q}=\{\alpha_1,\alpha_2,\cdots, \alpha_q\}$. For any integer $2\leq k\leq q$, the $[q, k-1]$ Reed-Solomon code $RS_q(k-1)$ over finite field $\f{q}$ is defined to be the $\f{q}$-vector space generated by rows of the $(k-1)\times q$ Vandermonde matrix 
$$	M(q,k-1)=\left(
\begin{array}{cccc}
1 & 1 & \cdots & 1  \\
\alpha_1 & \alpha_2 & \cdots & \alpha_q  \\
\vdots & \vdots & \ddots & \vdots  \\
\alpha_1^{k-2} & \alpha_2^{k-2} & \cdots & \alpha_q^{k-2} \\
\end{array}
\right).$$
It is an MDS code, which is equivalent to saying that 
every $(k-1)\times (k-1)$ submatrix of $M_q(k-1)$ has non-zero determinant.   

By Lagrange interpolation, any 
vector $\beta =(\beta_1,\cdots, \beta_q) \in \f{q}^q$ can be written uniquely as 
$$\beta = \beta_f:= (f(\alpha_1), \cdots, f(\alpha_q)),$$where 
$f(x)\in \f{q}[x]$ is a polynomial with $\deg(f)\leq q-1$. The word $\beta_f$ is a deep hole 
of the above Reed-Solomon code if and only if the row vectors of 
the following generalized $k\times q$ Vandermonde matrix 
$$	M_{f}(q,k-1)=\left(
\begin{array}{cccc}
1 & 1 & \cdots & 1  \\
\alpha_1 & \alpha_2 & \cdots & \alpha_q  \\
\vdots & \vdots & \ddots & \vdots  \\
\alpha_1^{k-2} & \alpha_2^{k-2} & \cdots & \alpha_q^{k-2} \\
f(\alpha_1) & f(\alpha_2) & \cdots & f(\alpha_q )\\
\end{array}
\right)$$
generate an MDS code, that is, every $k\times k$ submatrix of $M_f(q,k-1)$ has non-zero 
determinant. Equivalently, the determinant 
$$D_f(\alpha_{i_1},\cdots, \alpha_{i_k})=\det \left(
\begin{array}{cccc}
1 & 1 & \cdots & 1  \\
\alpha_{i_1} & \alpha_{i_2} & \cdots & \alpha_{i_k} \\
\vdots & \vdots & \ddots & \vdots  \\
\alpha_{i_1}^{k-2} & \alpha_{i_2}^{k-2} & \cdots & \alpha_{i_k}^{k-2} \\
f(\alpha_{i_1}) & f(\alpha_{i_2}) & \cdots & f(\alpha_{i_k} )\\
\end{array}
\right) \not= 0$$ 
for all $1\leq i_1<\cdots < i_k\leq q$. In this way, Theorem~\ref{thm5} is reduced to the degree classification of deep holes for Reed-Solomon codes.
Cheng-Murray \cite{CM07} conjectured that $\beta_f$ is a deep hole if and only if $\deg(f)=k-1$. 
This conjecture immediately implies (and in fact is equivalent) to Conjecture \ref{conj:zeroofcompsymmpoly}. 
This conjecture was already proven in the case $k=3\leq p$ for odd $q>5$ by Segre in his classical 
paper \cite{Seg55}. This special case is all we need to prove  Theorem \ref{thm1}. 

The Cheng-Murray Conjecture remains open in general, but has been proved by Zhuang-Cheng-Li \cite{ZCL16} 
in the case $k\leq p$  and later by Kaipa \cite{Kaipa17} in the case $k\geq [(q+1)/2]$.  As a consequence, 
Conjecture \ref{conj:zeroofcompsymmpoly}  is true if either $k\leq p$ or $k\geq [(q+1)/2]$.

%The results in \cite{ZCL16} and \cite{Kaipa17} depend crucially on results from finite geometry. 
To be self-contained, in the rest of this section, we include a simpler and more direct 
proof of these results motivated by the approach from \cite{Kaipa17}.  

It is well-known that the dual of a Reed-Solomon code is still a Reed-Solomon code. That is,
\begin{align}\label{dual:Vander}
M(q,k-1)M(q,q+1-k)^T=0.
\end{align} 
We present a proof to make it self-contained. For any polynomial $a(x)\in\f{q}[x]$ of degree $\leq k-2$ and any polynomial $b(x)\in\f{q}[x]$ of degree $\leq q-k,$ the product $a(x)b(x)$ has degree $\leq q-2$. By the Lagrange interpolation, we have
\begin{align*}
a(x)b(x)=\sum_{i=1}^{q}\frac{\prod_{j\neq i}(x-\alpha_j)}{\prod_{j\neq i}(\alpha_i-\alpha_j)}a(\alpha_i)b(\alpha_i)=\sum_{i=1}^{q}a(\alpha_i)b(\alpha_i){\prod_{j\neq i}(x-\alpha_j)}.
\end{align*}
The last equality holds because $\prod_{\alpha \in \f{q}^*}\alpha=1$. 
Comparing the terms of degree $q-1$ of both sides, we get
\begin{align}\label{dual:polynomial}
0=\sum_{i=1}^{q}a(\alpha_i)b(\alpha_i).
\end{align}
Taking $a(x)=x^{k_1}$ with $0\leq k_1 \leq k-2$ and $b(x)=x^{k_2}$ with $0\leq k_2 \leq q-k$, we deduce 
the orthogonality relation in 
Equation~(\ref{dual:Vander}).

Define the extended $k\times (q+1)$ matrix of $M_{f}(q,k-1)$ to be 
$$	M^E_{f}(q,k-1)=\left(
\begin{array}{ccccc}
1 & 1 & \cdots & 1 &0 \\
\alpha_1 & \alpha_2 & \cdots & \alpha_q  &0\\
\vdots & \vdots & \ddots & \vdots & \vdots \\
\alpha_1^{k-2} & \alpha_2^{k-2} & \cdots & \alpha_q^{k-2}&0 \\
f(\alpha_1) & f(\alpha_2) & \cdots & f(\alpha_q )&1\\
\end{array}
\right). $$
\begin{lem}\label{daul:extRS}
	Let $w_i=-\sum_{j=1}^{q}\alpha_j^{i}f(\alpha_j)$ for $i=0,1,\cdots,q-k,$ and
	 {\small
	\begin{align*}
	M^E_f(q,k-1)^\perp&=\left(
	\begin{array}{ccccc}
	1 & 1 & \cdots & 1 & w_0 \\
	\alpha_1 & \alpha_2 & \cdots & \alpha_q & w_1 \\
	\vdots & \vdots & \ddots & \vdots  & \vdots \\
	\alpha_1^{q-k} & \alpha_2^{q-k} & \cdots & \alpha_q^{q-k}& w_{q-k} \\
	\end{array}
	\right).
	\end{align*}}
Then we have
\begin{align}
  M^E_{f}(q,k-1)(M^E_f(q,k-1)^\perp)^T=0.
\end{align}
\end{lem}
\begin{proof}
	It follows from Equation~(\ref{dual:Vander}) that the rows of $M^E_f(q,k-1)^\perp$ are orthogonal to the first $k-1$ rows of $M^E_f(q,k-1).$ From the definition of $w_i$, we deduce that the rows of $M^E_f(q,k-1)^\perp$ are also orthogonal to the last row of $M^E_f(q,k-1).$ So we have $ M^E_{f}(q,k-1)(M^E_f(q,k-1)^\perp)^T=0.$
\end{proof}

Note that the linear code $C$ generated by rows of $M^E_f(q,k-1)^\perp$ is the extension of Reed-Solomon code $RS_q(q+1-k)$ by one digit. Seroussi and Roth considered the problem when the code $C$ preserves the MDS property.
\begin{lem}[{\cite[Theorem~1]{SR86}}]\label{ExtofNRC}
	If $\max\{3, \lfloor (q-1)/2\rfloor\}\leq k\leq q-2$ ($\max\{4,\lfloor (q-1)/2\rfloor\}\leq k\leq q-3$ if $q$ is even), every $q+1-k$ columns of $M^E_f(q,k-1)^\perp$ are linearly independent if and only if $$w=(w_0,w_1,\cdots,w_{q-k})^T=(0,0,\cdots,0,a)^T$$ for some $a\in\f{q}^*.$
\end{lem}

For small $k\leq p$, Ball gave the structure of MDS codes of length $(q+1)$ and dimension $q+1-k$.
\begin{lem}[{\cite[Theorem~1.10]{Ball12}}]\label{structure:q+1arc}
	For $3\leq k\leq p,$ every $q+1-k$ columns of $M^E_f(q,k-1)^\perp$ are linearly independent if and only if $$w=(w_0,w_1,\cdots,w_{q-k})^T=(0,0,\cdots,0,a)^T$$ for some $a\in\f{q}^*.$
\end{lem}

The following lemma is well-known as the property of MDS codes: the dual code of an MDS code is still an MDS code (see~\cite[Chapter~11]{book:MW}).
\begin{lem}\label{dual:MDS}
	The following two statements are equivalent:
	\begin{enumerate}
		\item every $k$ columns of $M^E_{f}(q,k-1)$ are linearly independent,
		\item every $q+1-k$ columns of $M^E_f(q,k-1)^\perp$ are linearly independent. 		
	\end{enumerate}
	\end{lem}

\subsection{The proof of Theorem~\ref{thm4}}
We have seen that Theorem~\ref{thm4} is equivalent to Theorem~\ref{thm5}. 
Furthermore, one direction of Theorem~\ref{thm5} already holds unconditionally on $k$. So we only need to prove that for any integer $k$ such that $3 \leq k \leq q-2$ ( $4\leq k\leq q-3$ for $q$ even) and for any polynomial $f(x)\in \f{q}[x]$ of degree $\deg(f)\leq q-1$,  if $N_q^*(D_f)=0$ then $\deg f(x)= k-1$ provided that $k\leq p$ or $k\geq \lfloor \frac{q+1}{2}\rfloor.$

 \begin{proof}
	Since $N_q^*(D_f)=0$, for any  pairwise distinct $\alpha_1,\cdots,\alpha_k\in \f{q}$, we have 
	$$D_f(\alpha_1,\cdots,\alpha_k)\neq 0.$$
	This condition is equivalent to that any $k$ columns of $M_f(q,k-1)$ are linearly independent,  which is also equivalent to that any $k$ columns of $M^E_f(q,k-1)$ are linearly independent. By Lemma~\ref{dual:MDS}, it is equivalent to that any $q+1-k$ columns of $M^E_f(q,k-1)^\perp$ are linearly independent. For $k\leq p$ by Lemma~\ref{structure:q+1arc} and for $k\geq \lfloor \frac{q+1}{2}\rfloor$ by Lemma~\ref{ExtofNRC}, we always have  $$w=(w_0,w_1,\cdots,w_{q-k})^T=(0,0,\cdots,0,a)^T$$ for some $a\in\f{q}^*.$
	 That is, we get a system of linear equations on variables $f(\alpha_1), f(\alpha_2), \cdots, f(\alpha_q)$:
			\begin{align}\label{non-hom.sol}
		\left(
		\begin{array}{cccc}
		1 & 1 & \cdots & 1 \\
		\alpha_1 & \alpha_2 & \cdots & \alpha_q \\
		\vdots & \vdots & \ddots & \vdots  \\
		\alpha_1^{q-k-1} & \alpha_2^{q-k-1} & \cdots & \alpha_q^{q-k-1}  \\
		\alpha_1^{q-k} & \alpha_2^{q-k} & \cdots & \alpha_q^{q-k} 
		\end{array}
		\right)	\left(
		\begin{array}{c}
		f(\alpha_1) \\
		f(\alpha_2)\\
		\vdots \\
	    f(\alpha_q)\\ 
		\end{array}
		\right)=	\left(
		\begin{array}{c}
		0 \\
		0\\
		\vdots \\
		0\\
		a\\ 
		\end{array}
		\right).
		\end{align}
		Note that 
		\begin{itemize}
			\item by Equality~(\ref{dual:Vander}), 
			to satisfy the first $q-k$ equations, the vector $$(f(\alpha_1), f(\alpha_2), \cdots, f(\alpha_q))$$ must belong to 
			the $\f{q}$-linear vector space generated by rows of 
			$$	M(q,k)=\left(
			\begin{array}{cccc}
			1 & 1 & \cdots & 1  \\
			\alpha_1 & \alpha_2 & \cdots & \alpha_q  \\
			\vdots & \vdots & \ddots & \vdots  \\
			\alpha_1^{k-1} & \alpha_2^{k-1} & \cdots & \alpha_q^{k-1} \\
			\end{array}
			\right). $$
			\item This means that there is a polynomial $g(x) \in \f{q}[x]$ of degree at most $k-1$ such that $f(\alpha_i) =g(\alpha_i)$ for all 
			$1\leq i\leq q$. Since $\deg(f)\leq q-1$, it forces that $f(x)=g(x)$, which has degree at most $k-1$. On the other hand, if $\deg f(x)\leq k-2,$ then it is easy to see that $D_f$ is the zero polynomial which contradicts to the assumption $N_q^*(D_f)=0$. So $\deg f(x)=k-1.$ The proof is complete. 
		\end{itemize}

\end{proof}

\begin{rem}
	Conjecture~22 in~\cite{BL19} says that for $6\leq k\leq q-5$, any MDS code of length $(q+1)$ and dimension $k$ is equivalent to the Reed-Solomon code $RS_q(k)$ up to permutation of the coordinates and multiplication of symbols appearing in a fixed coordinate by a non-zero scalar. This conjecture together with the proof above imply  Conjecture~\ref{conj:zeroofcompsymmpoly} for $6\leq k\leq q-5$. 
\end{rem}

\section{A generalization}
In this section, we shall work in the more general framework of generalized Vandermonde determinant problem. This corresponds to solving complete symmetric polynomials on a subset of the finite field. 

Let $2\leq k\leq n\leq q$ be integers. 
Let $S\subset \f{q}$ be a subset of cardinality $n$. For any polynomial $f(x)\in \f{q}[x]$ let 	
$$N_{S}^*(D_f):=\# \{(\alpha_1,..., \alpha_k)\in S^k\,|\, D_f(\alpha_1,..., \alpha_k)=0\,~\mbox{and}~\, 
\alpha_i\not= \alpha_j ~\forall ~i\not=j\}.$$

This general framework is motivated by the following example. If $S\subsetneq \f{q}$ is a proper subset, then for any $\alpha\in \f{q}\setminus S$,  we have 
$(\alpha_i-\alpha)^{q-2} = (\alpha_i-\alpha)^{-1}$ and thus 
\[
D_{(x-\alpha)^{q-2}}(\alpha_1,\cdots,\alpha_k)=\frac{1}{\prod_{i=1}^{k}(\alpha_i-\alpha)}\prod_{1\leq i<j\leq k}(\alpha_j-\alpha_i)\neq 0,
\]
for any pairwise distinct $\alpha_1,\cdots,\alpha_k\in S.$ So Theorem~\ref{thm5} does not hold anymore if we consider zeros of $D_f$ in a proper subset $S\subset \f{q}$.

\begin{prob}\label{vanderprob}
	For a given subset $S\subset\f{q}$ and polynomial $f(x)\in \f{q}[x]$ with $\deg(f)\leq q-1$, is there any efficient way (e.g. polynomial time in $\log q$, $|S|$ and $\deg(f)$) to determine if $D_f(\alpha_1,\cdots,\alpha_k)\neq 0$ for all pairwise distinct $\alpha_1,\cdots,\alpha_k\in S?$
\end{prob}

This algorithmic problem is difficult in such a generality. In fact, we will soon see that this problem is NP-hard for general $S$. 
For even $q$, $k=3$ and $S=\f{q}$, the problem is the classification of hyperovals in finite projective plane $\mathbb{P}^2(\f{q})$, which is still open (see~\cite[Section~14.1]{BL19} for the collection of known families of hyperovals).

The brute-force algorithm takes time 
\[
\binom{|S|}{k}\times (\mbox{time of computing the determinant of $k\times k$ matrix)},
\]
which is exponential in $|S|$ when $k=c|S|$ for any $0<c<1.$

Here are some examples of low degrees by Proposition~\ref{fink}.
\begin{enumerate}
	\item In the case $d=k-1$, $f(x)=\sum_{i=0}^{k-1}a_ix^i$\,\,($a_{k-1}\neq 0$), so
	\[
	C_f(x_1, \cdots, x_k)=a_{k-1}\,\,\mbox{and}\,\,
	D_f(x_1,\cdots, x_k)=a_{k-1}\prod_{1\leq i<j\leq k}(x_i-x_j).
	\]
	Hence, $D_f(\alpha_1,\cdots,\alpha_k)\neq 0$ for any  pairwise distinct $\alpha_1,\cdots,\alpha_k\in S.$
	\item In the case $d=k$, $f(x)=\sum_{i=0}^{k}a_ix^i$\,\,($a_k\neq 0$), so
	\[
	C_f(x_1, \cdots, x_k)=a_{k-1}+a_k(x_1+\cdots+ x_k)
	\]
	is linear and
	\[
	D_f(x_1, \cdots, x_k)=(a_{k-1}+a_k(x_1+\cdots+ x_k))\prod_{1\leq i<j\leq k}(x_i-x_j).
	\]
	Hence, for any  pairwise distinct $\alpha_1,\cdots,\alpha_k\in S$, we have $D_f(\alpha_1,\cdots, \alpha_k)= 0$ if and only if $a_{k-1}+a_k(\alpha_1+\cdots+ \alpha_k)=0.$
	This is exactly the $k$-subset sum problem ($k$-SSP) over $S$ which is known to be \textbf{NP-complete} for general $S$. 
	For special $S$, e.g. $S=\f{q}$ or $\f{q}^*$, there is an explicit formula for $N_q^*(C_f)$,  see~\cite{LW08}, which implies that 
	$N_q^*(D_f)>0$ for $3\leq k \leq q-2$. 
	\item In the case $d=k+1$, $f(x)=\sum_{i=0}^{k+1}a_ix^i$\,\,($a_{k+1}\neq 0$), so
	\[
	C_f(x_1, \cdots, x_k)=a_{k-1}+a_k\sum_{i=1}^kx_i+a_{k+1}\left(\sum_{i=1}^kx_i^2+\sum_{1\leq i< j\leq k}x_ix_j\right).
	\]
	is quadratic. It was shown in~\cite[Theorem~4.2]{ZFL13} that $N_q^*(D_f)>0$ for $3\leq k \leq q-2$ ($k \not= q-2$ if $q$ is even).	
\end{enumerate}

Conjecture \ref{conj:zeroofcompsymmpoly} is equivalent to the following conjecture. 

\begin{conj}\label{conj:wholefield} Let $3\leq k\leq q-2$ ($4\leq k\leq q-3$ if $q$ is even).	
	For any polynomial $f(x)\in \f{q}[x]$ of degree $ k\leq \deg(f)\leq q-1$, 
	there exist pairwise distinct $\alpha_1,\cdots,\alpha_k\in \f{q}$ such that $$D_f(\alpha_1,\cdots,\alpha_k)=0.$$ 
\end{conj}

This conjecture answers Problem~\ref{vanderprob} when $S=\f{q}$: $D_f(\alpha_1,\cdots,\alpha_k)\neq 0$ for all pairwise distinct $\alpha_1,\cdots,\alpha_k\in \f{q}$ if and only if $\deg(f)=k-1.$ 

Note that the conjecture is false if we restrict $\alpha_1,\cdots,\alpha_k$ in a proper subset $S$ of $\f{q}$. Suppose $\alpha\in \f{q}\setminus S,$ taking $f(x)=(x-\alpha)^{q-2}$, we have $D_f(\alpha_1,\cdots,\alpha_k)\neq 0$ for all pairwise distinct $\alpha_1,\cdots,\alpha_k\in S.$ 
If $|S|=q-1$, by a translation we can assume that $S=\f{q}^*$. In this case, we have the 
following similar conjecture. 

\begin{conj}\label{conj:nonzeros} Let $3\leq k\leq q-2$ ($4\leq k\leq q-3$ if $q$ is even). 
	For any polynomial $f(x)\in \f{q}[x]$ with $ k\leq \deg(f)\leq q-2$,  except those of the form $ax^{q-2}+g(x)$ for some $a\neq 0$ and polynomial $g(x)\in \f{q}[x]$ of degree $\deg(g)\leq k-2$, there exist pairwise distinct $\alpha_1,\cdots,\alpha_k\in \f{q}^*$ such that $$D_f(\alpha_1,\cdots,\alpha_k)=0.$$ 
\end{conj}

A similar proof of Theorem~\ref{thm5} together the original theorem~\cite[Theorem~1]{SR86} of Seroussi and Roth can show Conjecture~\ref{conj:nonzeros} is true if $k \geq (q+1)/2$.

\bibliographystyle{plain}
\bibliography{rationalpoints}

\end{document}